\documentclass[11 pt]{article}

\usepackage[latin1]{inputenc}
\usepackage{amsmath,amsfonts,stmaryrd,amsthm, enumitem, graphicx}

\usepackage{hyperref}

\usepackage{fullpage}
\usepackage{changes, comment}
\usepackage{soul}

\newtheorem{thm}{Theorem}
\newtheorem{lemma}[thm]{Lemma}
\newtheorem{prop}[thm]{Proposition}
\newtheorem{cor}[thm]{Corollary}

\theoremstyle{definition}
\newtheorem{definition}{Definition}
\newtheorem{rem}{Remark}

\newcommand{\cH}{\mathcal{H}}
\newcommand{\cI}{\mathcal{I}}
\newcommand{\cJ}{\mathcal{J}}

\makeatletter
\renewcommand*{\@fnsymbol}[1]{\ifcase#1\or$\heartsuit$\else$\clubsuit$\fi}
\makeatother

\title{On the number of coloured triangulations of $d$-manifolds}

\author{Guillaume Chapuy%
\thanks{This project has received funding from the European Research Council
(ERC) under the European Union's Horizon 2020 research and innovation 
programme (grant agreement No. ERC-2016-STG 716083 ``CombiTop'').
Email:~{\tt guillaume.chapuy@irif.fr}.
}
\\
\footnotesize\it CNRS, IRIF UMR 8243,\\
\footnotesize\it Universit\'e de Paris,\\
\footnotesize\it France.\\
\and Guillem Perarnau%
\thanks{
Email:~{\tt guillem.perarnau@upc.edu}.
G.P. acknowledges an invitation in Paris funded by the ERC grant CombiTop, during which this project was advanced.
}
\\
\footnotesize\it Departament de Matem\`atica Aplicada.\\
\footnotesize\it Universitat Polit\`ecnica de Catalunya,\\
\footnotesize\it Barcelona, Spain.\\
}

\renewcommand{\ll}{\preceq}

\begin{document}
\maketitle

\begin{abstract}
We give superexponential lower and upper bounds on the number of coloured $d$-dimensional triangulations whose underlying space is an oriented manifold, when the number of simplices goes to infinity and $d\geq 3$ is fixed.
	In the special case of dimension $3$, the lower and upper bounds match up to exponential factors, and we show that there are $2^{O(n)} n^{\frac{n}{6}}$ coloured triangulations of $3$-manifolds with $n$ tetrahedra. Our results also imply that random coloured triangulations of $3$-manifolds have a sublinear number of vertices.

Our upper bounds apply  in particular to coloured $d$-spheres for which they seem to be the best known bounds in any dimension $d\geq 3$, even though it is often  conjectured that exponential bounds hold in this case.

	We also ask a related question on regular edge-coloured graphs having the property that each $3$-coloured component is planar, which is of independent interest.

\end{abstract}

\section{Introduction and main results}

A famous question, sometimes attributed to Gromov~\cite{rivasseau2013spheres, gromov2010spaces} but going back at least to Durhuus and J\'onsson~\cite{DJ}, asks whether for any dimension $d\geq 2$ the number of inequivalent triangulations of the $d$-sphere by $n$ unlabelled simplices is bounded by $K^n$ for some constant $K=K(d)$. 
In dimension $d=2$, it is not difficult to see that the answer is yes by noticing that each planar triangulation can be encoded by a spanning tree and a parenthesis word (one can also use the explicit formula due to Tutte~\cite{Tutte}). In dimension $d\geq3$, this question is open. In the pioneering paper~\cite{DJ}, Durhuus and J\'onsson introduced a subclass of triangulated spheres called locally constructible, or \emph{LC}, and showed that their number is exponentially bounded. They conjectured that all $3$-spheres are LC, but this was disproved many years later by Benedetti and Ziegler in~\cite{BenedettiZiegler}. Other subclasses of spheres with exponential growth have been considered~\cite{Mogami, BenedettiOther,DJcausal, ColletEckmannYounan, AnoBenedetti}, but the question remains wide open.

A motivation for the study of triangulations comes from the discretization of space in quantum gravity, see~\cite{QuantumGeometry, rovelli2007quantum}. Recently there has been a renewed interest in this topic via coloured tensor models, which are a higher dimensional generalization of matrix integrals, see~\cite{Gur16,Carrance}. The objects that naturally arise in these models are \emph{coloured triangulations}: roughly speaking, a triangulation is \emph{coloured} if the vertices are coloured with colours from $1$ to $d+1$, and all colours appear in each $d$-simplex. These triangulations are defined by gluing in pairs the facets ($(d-1)$-dimensional faces) in a family of $n$ abstract coloured $d$-simplices, and need not be simplicial -- precise definitions are given below. Because each triangulation can be made coloured by an appropriate barycentric subdivision that multiplies the number of simplices by a  factor that only depends on $d$ (see e.g. Section 2 in~\cite{rivasseau2013spheres}), the answer to the ``Gromov question'' is the same for coloured and uncoloured triangulations. In this paper we will work only with coloured triangulations, that are nicer and more natural as combinatorial objects.

\smallskip
In~\cite{rivasseau2013spheres}, Rivasseau showed that the number of coloured triangulations of the $d$-sphere with $n$ \emph{labelled} simplices grows at most like $n^{\left(\frac{d-1}{2}+\frac{1}{d}\right)n}$ up to exponential factors of the form $K^n$. This bound improves the trivial bound $n^{\left(\frac{d+1}{2}\right)n}$ which counts \emph{all} complexes obtained by arbitrary gluings of $n$ labelled $d$-simplices along their facets, up to exponential factors. Equivalently there are at most $n^{\left(\frac{d-3}{2}+\frac{1}{d}\right)n}$ inequivalent (unlabelled) coloured triangulations of the $d$-sphere with $n$ $d$-simplices up to exponential factors, which as far as we know was prior to this work the best known upper bound in the direction of Gromov's question.

This naturally raises the question of improving further the constants driving this superexponential growth.
In this paper we address this question, but under a weaker topological constraint: instead of considering $d$-spheres, we consider triangulations whose underlying space is a $d$-manifold. Since spheres are manifolds, the  superexponential upper bounds that we obtain for $d$-manifolds apply in particular to $d$-spheres, and in fact they improve the ones of~\cite{rivasseau2013spheres}. We also give superexponential lower bounds  obtained by explicit constructions (of course these lower bounds do not apply to $d$-spheres). Our lower and upper superexponential bounds match in dimension $d=3$, but a gap remains in higher dimension. Our main result, Theorem~\ref{thm:main} below, summarizes these results.

\subsection{Main results}
Here and everywhere in this paper, a \emph{manifold} is a topological manifold, that is to say a Haussdorff topological space that is locally homeomorphic to $\mathbb{R}^n$ (we do not consider manifolds with boundary). All the topological spaces we consider will be compact.

For $d\geq 2, n\geq 2$, we let $M_d(n)$ be the number of coloured $d$-dimensional triangulations of orientable manifolds, with $n$ $d$-simplices labelled from $1$ to $n$ (see formal definition in Section~\ref{sec:defs}). As we will see this number is zero if $n$ is odd as a consequence of orientability, and throughout the paper $n$ will be assumed to be even.

We will mostly be interested in the superexponential growth of the sequences we consider, {\it i.e.} we will
often disregard factors of the form $K^n$, and for this we will use the following notation 
\renewcommand{\ll}{\preceq}
\begin{eqnarray}\label{defsymb}
	f(n) &\ll& g(n) \mbox{ iff } \exists K>0 \mbox{ such that for $n$ large enough, } f(n) \leq K^n g(n) \\
	f(n) &\asymp& g(n) \mbox{ iff } f(n)\ll g(n) \mbox{ and } g(n)\ll f(n).
\end{eqnarray}
Moreover, in all asymptotic statements in this paper, it is implicitly assumed that $d$ is constant and $n$ goes to infinity with $n$ being even.

\begin{thm}[Main result]\label{thm:main}
	 For all $d\geq 3$ with $d\neq 4$ we have
		\begin{align}\label{eq:main}
			\displaystyle n^{\frac{n}{2d}} \ll \frac{1}{n!}M_d(n)  \ll n^{\left(\frac{1}{6}+(d-3) \frac{3}{20}\right) n},
		\end{align}
	and for $d=4$ we have 
		$$\displaystyle	n^{\frac{n}{8}} \ll \frac{1}{n!}M_4(n)  \ll n^{\frac{n}{3} }.$$
\end{thm}
Since the upper and lower bounds in~\eqref{eq:main} match in dimension $d=3$, we obtain:
\begin{cor}\label{cor:dim3}
	The number of coloured triangulations of orientable $3$-manifolds with $n$ labelled tetrahedra satisfies $$\displaystyle \frac{1}{n!}M_3(n)  \asymp n^{\frac{n}{6}}.$$
\end{cor}

These results raise the question of determining the constant
$$
\displaystyle
\alpha_d := \limsup_{n\to \infty} \frac{\log ( M_d(n)/n! )}{n\log n}.
$$
Our main theorem shows that $\alpha_3=1/6$, $\alpha_4\in[\tfrac{1}{8},\tfrac{1}{3}]$, and for $d\geq 5$ 
\begin{align}\label{eq:alpha}
	\frac{1}{2d}\leq \alpha_d\leq \frac{1}{6}+ (d-3)\frac{3}{20},
\end{align}
which leaves an important gap especially for large $d$. We believe that the lower bound is closer to the truth, and we conjecture that $\alpha_d$ goes to zero when $d$ goes to infinity.

\subsection{Additional results, related work and comments}

In the uncoloured setting, the fact that the number of triangulated $3$-manifolds with $n$ unlabelled tetrahedra grows \emph{at least} as $n^{\epsilon n}$ for explicit positive values of $\epsilon$ was proved in~\cite{ambjorn1991three}~(see also~\cite{benedetti2010locally}). Their construction is based on Heegaard gluings of high genus triangulations and inspired our general lower bound construction. Using the distribution of short cycles in the configuration model~\cite{wormald1999models}, in~\cite{dunfield2006finite} it was proved that the probability the underlying space of a $3$-dimensional triangulation  with $n$ unlabelled tetrahedra is a $3$-manifold is~$o(1)$. 
It may be possible to make the argument in~\cite{dunfield2006finite} quantitative in order to obtain an upper bound for the number triangulated $3$-manifolds of the form $n^{(1-\epsilon)n}$ for an explicit $\epsilon>0$. However, the uncoloured case seems combinatorially more involved than the coloured one, and this would probably not lead to a sharp value of $\epsilon$ as the one we have here in the coloured case. 

In fact our upper bounds extend to a much larger class of triangulations than $d$-manifolds, roughly speaking to triangulations whose  residues with a small number of colours are spheres.
See Theorem~\ref{thm:upperBounds} on Section~\ref{sec:UB} for our most general upper bounds. 
Moreover for (manifold) homology $d$-spheres of dimension $d=4$, we obtain a slightly better upper bound than the one we have for $M_4(n)$.

Finally, it is natural to ask if our enumerative results have probabilistic consequences, especially in dimension $d=3$ where our upper and lower bounds match. Because they match only up to exponential factors, the only events that there is hope to control are the ones with exponentially small probability. Indeed we can show, as a corollary of our proofs:
\begin{thm}\label{thm:proba}
Consider a coloured $3$-dimensional triangulation with $n$ labelled tetrahedra whose underlying space is an orientable $3$-manifold, chosen uniformly at random among such objects. Let $V_n$ be its number of vertices.		Then $V_n = O(\frac{n}{\log n})$ with high probability. 
			
			More precisely, for all $c>0$ there exists $K$ such that for all $n\geq 1$,
			$$
	\mathbf{Pr}\left(V_n \leq\frac{K n}{\log n} \right) \geq 1-e^{-cn}.
	$$
\end{thm}
It is natural to expect that the number of vertices is in fact $O(\log n)$ with high probability as in the 2-dimensional case, but our techniques do not enable to prove it.

%

\section{Colourful graphs and triangulations}

\subsection{Definitions}
\label{sec:defs}
\begin{definition}
	For $d\geq 1$ and $n$ even, a \emph{(d+1)-colourful graph} of order $n$ is a bipartite $(d+1)$-regular multigraph on $[1..n]$, equipped with a colouring of its edges with colours in $[1..d+1]$, such that each vertex is incident to all colours.	
\end{definition}

We equip the colourful graph with a colouring of its vertex set, where one part of the bipartition is assigned colour \emph{white} and the other one, colour \emph{black}.

\begin{figure}
	\centering
		\includegraphics[width=0.6\linewidth]{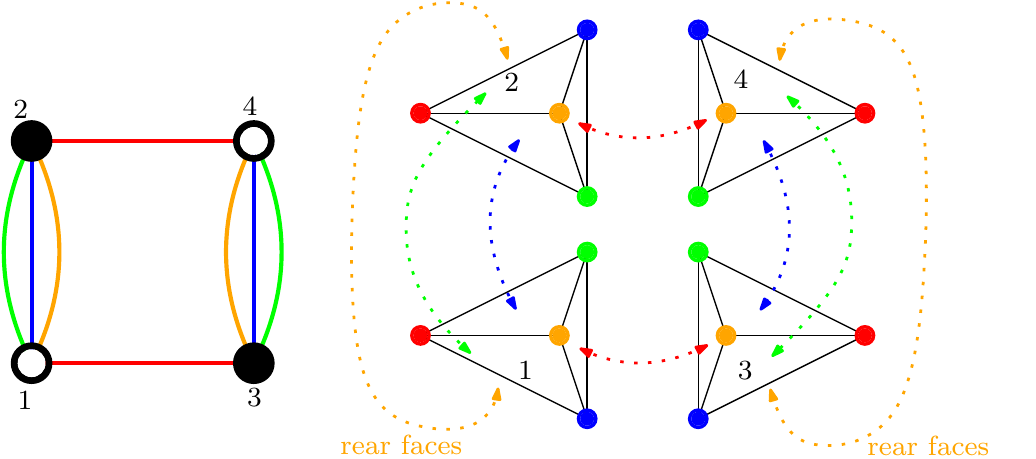}
		\caption{A $(d+1)$-colourful graph $G$ with $d=3$ and $n=4$, and the corresponding cell complex. The resulting space $|X(G)|$ is a $3$-sphere -- an easy way to see that is to perform first the six non-red gluings, which clearly gives two disjoint $3$-balls, that are then glued along their boundary via the red gluings. Note that the complex $X(G)$ is \emph{not} simplicial, since the two tetrahedra on each side share the same set of vertices. Since the sphere is a manifold,  this object contributes to the number $M_3(4)$.}
\end{figure}

Given a $(d+1)$-colourful graph $G$, we can construct a topological object as follows.
For each  $v\in [1..n]$, we consider an abstract $d$-simplex, and we colour its vertices from $1$ to $(d+1)$.
Apart from these colours, the vertices of each simplex are unlabelled. We see each simplex as a solid body being a copy of the regular $d$-simplex $\{x_1+\dots+x_{d+1}=1, x_i\geq0\}$ equipped with its Euclidean topology. Now for each edge $e=\{u,v\}$ of $G$ of colour $i$, we consider the unique facet in each of the two simplices corresponding to $u$ and $v$, whose vertices are coloured by $[1..d\!+\!1]\setminus \{i\}$. We glue these two facets together according to the unique isometric gluing that preserves colours, and we repeat this procedure for each edge of $G$.

We call $X(G)$ the corresponding cell complex. Complexes obtained in this way are called \emph{labelled $d$-dimensional coloured triangulations}. We denote by $|X(G)|$ the resulting topological space, and we observe that it is orientable because $G$ is bipartite \footnote{The space $|X(G)|$ can also be defined for non-bipartite analogues of colourful graphs,  giving rise to triangulated non-orientable topological spaces.}.
In fact, $d$-dimensional coloured triangulations of orientable topological spaces with $n$ simplices are in bijection with $d$-colourful graphs of order $n$, as one can recover the graph $G$ from the triangulation $X(G)$ by taking its dual graph, in which vertices correspond to highest dimensional cells and edges to $(d-1)$-dimensional incidences between these cells.
In particular, such objects exist only if $n$ is even, as any regular bipartite graph has an even number of vertices.

While for $d=2$ the space $|X(G)|$ is always a manifold, we emphasize that if $d\geq 3$ this is not always the case.
To see this, notice that the link of a given vertex in $X(G)$ can be any connected coloured $(d-1)$-dimensional triangulation, and therefore for $d-1>1$ it can have nontrivial homology, contradicting Criterion $(iii)$ in Proposition~\ref{prop:criteria} below.

The bijection described above gives an alternative definition for the number of coloured $d$-dimensional triangulations with $n$ $d$-simplices labelled from $1$ to $n$ whose underlying space is an orientable manifold.
\begin{definition}
	We let $M_d(n)$ be the number of $(d+1)$-colourful graphs $G$ on $[1..n]$ such that $|X(G)|$ is a $d$-manifold.
\end{definition}

 \subsection{Homology and residues}

Let $G$ be a $(d+1)$-colourful graph and $\cI\subseteq [1..d+1]$. We let  $G_\cI$ be the graph on the same vertex set as $G$, keeping only the edges whose colour is in $\cI$. Up to relabelling colours keeping their relative natural order, $G_\cI$ is an $|\cI|$-colourful graph associated with some coloured triangulation $X(G_\cI)$ of dimension $|\cI|-1$.

It is easy to see that connected components of $G_\cI$ are in bijection with $(d+1-|\cI|)$-dimensional cells of $X(G)$ whose colours do not belong to $\cI$.
For a proof, see~\cite{BrachoMontejano} in which these connected components are called~\emph{residues} of $G$.
For example, for each $i\in [1..d+1]$, vertices of $X(G)$ of colour $i$ are in bijection with connected components of $G_\cI$ for $\cI=[1..d+1]\setminus \{i\}$.

In this paper we will work with singular homology of topological spaces.
Given a $(d+1)$-colourful graph $G$, we say that $|X(G)|$ is a \emph{rational homology sphere of dimension $d$} if it has the same homology groups over the field of rationals as a $d$-sphere. Equivalently, all its Betti numbers are zero, except the $0$-th and $d$-th ones that are equal to one. We insist on the fact that no other property is required in this definition, in particular a rational homology sphere is not necessarily a manifold. Note that each integral homology sphere is a rational homology sphere, but the converse is not true in general.

\begin{prop}\label{prop:criteria}
        Let $G$ be a $(d+1)$-colourful graph, and consider the three following properties
        \begin{itemize}
        \item[(i)]
                        For any $I \subseteq [1..d+1]$ such that
 $1\leq |\cI|\leq d$,  $|X(G_\cI)|$ is a disjoint union of spheres (of dimension $|\cI|-1$);
                                \item[(ii)]
                        $|X(G)|$ is a $d$-manifold;
                \item[(iii)]
                        For any $I \subseteq [1..d+1]$ such that
			$1\leq |\cI|\leq d$,   $|X(G_\cI)|$ is a disjoint union of rational homology spheres (of dimension $|\cI|-1$).
        \end{itemize}
        Then one has: $(i)\Rightarrow (ii)\Rightarrow (iii)$.
        \end{prop}
\begin{proof}
        Let $H$ be a connected component of $G_{\cI}$, and let $\sigma$ be the $\cI^c$-coloured cell of $X(G)$ corresponding to $H$. Let $b_\sigma$ be the barycentre of $\sigma$, which is a vertex in the barycentric subdivision $X(G)'$ of $X(G)$. It is easy to check (see \cite[Proposition 2.5]{BrachoMontejano} and the sentence before it for a proof) that the topological join $|X(H)|*(\delta \sigma)'$ is the link of $b_\sigma$ in $X(G)'$ (in the case where $|\cI|=d$, i.e. $\sigma$ is a $0$-simplex, we conventionally understand this join as being just $|X(H)|$; in the case where $|\cI|=1$, i.e. $H$ is just one edge, we conventionally understand $|X(H)|$ as the disjoint union of two isolated vertices).

        Therefore the implication $(i)\Rightarrow (ii)$ is clear. Let us prove $(ii)\Rightarrow (iii)$.
        Let $U=|X(H)|$ and $V=(\delta \sigma)'$, we have the isomorphism 
$$
	\tilde H_{k}(U\times V)\equiv \tilde H_{k+1}(U*V) \oplus \tilde H_k(U) \oplus \tilde H_k(V),
$$
        where $k\geq 0$ and $\tilde H_k$ is the $k$-th reduced homology group over the rationals. By taking ranks, multiplying by $x^{k}$ and summing over $k$, we get
$$
        P_{U\times V}(x)-1 = x^{-1} \big( P_{U* V}(x)-P_{U*V}(0) \big) + P_U(x)-1 +P_V(x)-1,
$$
        where $P_X$ denotes the Poincar\'e polynomial of the topological space $X$.
        Now if $|X(G)|$ is a manifold, so is $|X(G)'|$, and since $X(G)'$ is a \emph{simplicial} complex, the link $U*V$ is a homology sphere, see e.g.~\cite[Prop. 5.2.4]{shastri2016basic}. We thus have $P_{U*V}(x)=1+x^{d-1}$, and since $P_{U \times V}=P_U P_V$  we finally obtain
        $$(P_U(x)-1)(P_V(x)-1) = x^{d-2}.$$
        Since $P_U=P_{|X(H)|}$ has degree $|\cI|-1$, it follows that $P_U(x)-1=x^{|\cI|-1}$, i.e. $|X(H)|$ is a rational homology sphere.
\end{proof}

\begin{rem}\label{rem:planarLinks}
	In the special case $|\cI|=3$, say $\cI=\{i,j,k\}$ with $i<j<k$, the graph $G_\cI$ is a 3-edge coloured bipartite cubic graph. We can define an embedding of $G_\cI$ in an orientable surface by defining the clockwise order of edges to be $(i,j,k)$ around white vertices and $(i,k,j)$ around black vertices. This enables us to view the graph $G_\cI$ as an \emph{embedded} graph, and it is easy to see that it is the dual of the complex $X(G_\cI)$, which is a surface  triangulation. In particular, if $|X(G)|$ is a manifold, then for any $\cI$ with $|\cI|=3$ the complex $X(G_\cI)$ is a disjoint union of spherical triangulations by Proposition~\ref{prop:criteria}, and each connected component of the embedded graph $G_\cI$ is a plane graph -- with the canonical embedding just defined.
\end{rem}



\subsection{A question about 3-planar colourful graphs}

Let $\cH_d(n)$ be the class of $(d+1)$-colourful graphs on $[1..n]$ having the following property:
	\begin{quote}\centering\it
		(P): For any subset of colours $\cI\subseteq [1..d+1]$ such \\ 
		~ ~ ~ ~~ that $|\cI|=3$, the embedded graph $G_\cI$ is planar. 
			\end{quote}
	From Remark~\ref{rem:planarLinks}, colourful graphs associated to manifolds satisfy property $(P)$, and this fact is important in the proof of our upper bounds.

	We do not expect $H_d(n):=|\cH_n(d)|$ and $M_d(n)$ to have the same superexponential growth, and in fact the proof of our main upper bound uses more than property $(P)$.
However determining the growth of $H_d(n)$ is a purely graph-theoretic question of independent interest. Let
$$\displaystyle
\beta_d := \limsup_{{n\to \infty}} \frac{\log ( H_d(n)/n! )}{n\log n}.
$$
Our proofs show that $\beta_3=\frac{1}{6}$ and that for $d\geq 4$
\begin{align}\label{eq:upperBoundH}
	\frac{1}{6} \leq \beta_d \leq \frac{d-2}{6}.
\end{align}
Determining $\beta_d$ is an interesting problem on its own. Moreover, since $\alpha_d\leq \beta_d$, a substantial improvement of the upper bound on $\beta_d$, would lead to an improvement of our main result. The problem of determining $\beta_d$ seems much more tractable {\it a priori}  than the one of determining $\alpha_d$. On the other hand, we believe that $\alpha_d$ and $\beta_d$ have different asymptotic behaviour, so it is unlikely that one can obtain a tight upper bound on $\alpha_d$ using $\beta_d$, for $d\geq 4$.



\section{Upper bounds}

\subsection{Main lemmas}

We now fix a $(d+1)$-colourful graph $G$ on $[1..n]$. For $\cJ\subseteq [1..d+1]$ we let $\kappa_\cJ$ be the number of connected components of the graph $G_\cJ$. For small values of $|\cJ|$ we drop brackets in the notation, for example $\kappa_{i,j}=\kappa_{\{i,j\}}$.

Let $\hat{d}$ be a number in $[1..d]$. For $\mathcal{I}\subseteq[{1}..d+1]$ with $|\cI|=\hat{d}+1$ and for $r \in [0..\hat{d}+1]$,  we let 
$$
\kappa^{(r)}_\cI:=\sum_{\cJ\subseteq \cI\atop |\cJ|=r} \kappa_\cJ\;.
$$
Note that $\kappa^{(r)}_{\cI}$ is the number of cells of $X(G_\cI)$ of dimension $(\hat{d}-r)$. In particular, $\kappa^{(0)}_{\cI}=n$ is the number of $\hat{d}$-simplices. Note also that, because $G_\cI$ is $(\hat{d}+1)$-regular, we have $\kappa^{(1)}_{\cI} = \frac{\hat{d}+1}{2}n$.

Moreover, if $\hat{d}$ is even and if $|X(G_\mathcal{I})|$ is a disjoint union of rational homology $\hat{d}$-spheres, the Euler-Poincar\'{e} formula states that
\begin{align}\label{eq:EP}
\sum_{r=0}^{\hat{d}} (-1)^{r} \kappa_\cI^{(r)} = 2\kappa_\cI^{(\hat{d}+1)}\;
\end{align}

\begin{lemma}\label{lemma:3count}
	Let $\cI \subseteq [1..d+1]$ with $|\cI|=3$ such that $|X(G_\cI)|$ is a disjoint union of $2$-spheres. Then there exist distinct $i,j \in \cI$ such that
	\begin{align}\label{eq:3count}
		\kappa_{i,j}- \kappa_\cI \leq \frac{n}{6}.
	\end{align}
\end{lemma}
\begin{proof}
The Euler-Poincar\'e formula~\eqref{eq:EP} for $\hat{d}=2$ implies
$$
\kappa^{(0)}_\cI+\kappa^{(2)}_\cI= \kappa^{(1)}_\cI+2\kappa^{(3)}_\cI\;.
$$
Since $\kappa^{(0)}_\cI=n$ and $\kappa^{(1)}_\cI=3n/2$, we obtain
$$
\kappa^{(2)}_\cI = 2\kappa^{(3)}_\cI +\frac{n}{2} =  2\kappa_\cI +\frac{n}{2}\;.
$$
By averaging over pairs of colours in $\cI$, there exist distinct $i,j\in \cI$ such that 
\begin{align}\label{eq:last}
\kappa_{i,j}\leq \frac{2\kappa_\cI}{3} + \frac{n}{6} \leq \kappa_\cI + \frac{n}{6}\;.
\end{align}
\end{proof}

\begin{lemma}\label{lemma:4count}
	Let $\cI \subseteq [1..d+1]$ with $|\cI|=5$ such that $|X(G_\cI)|$ is a disjoint union of rational homology $4$-spheres. Then there exist distinct $i,j,k\in \cI$ such that
	\begin{align}\label{eq:4count}
		\kappa_{i,j}- \kappa_{i,j,k} \leq \frac{3}{20}n.
	\end{align}
\end{lemma}
\begin{proof}
The Euler-Poincar\'e formula~\eqref{eq:EP} for $\hat{d}=4$ implies
$$
	\kappa^{(0)}_\cI+\kappa^{(2)}_\cI+\kappa^{(4)}_\cI= \kappa^{(1)}_\cI+\kappa^{(3)}_\cI+{2}\kappa^{(5)}_\cI\;, 
$$
	Since by deleting a colour, the number of connected components can only increase, we have $\kappa^{(4)}_\cI \geq 5\kappa^{(5)}_\cI$ and in particular $\kappa^{(4)}_\cI\geq {2}\kappa^{(5)}_\cI$.
	Since $\kappa^{(0)}_\cI=n$ and $\kappa^{(1)}_\cI=5n/2$, we obtain
$$
\kappa^{(2)}_\cI \leq  \kappa^{(3)}_\cI +\frac{3n}{2}\;,
$$
or equivalently,
$$
\frac{1}{3}\sum_{i,j,k\in \cI}(\kappa_{i,j}+\kappa_{j,k}+\kappa_{k,i}) \leq  \sum_{i,j,k\in \cI} \left(\kappa_{i,j,k} +\frac{3n}{20}\right)\;.
$$
Thus, there exists a triple of distinct $i,j,k\in \cI$ satisfying
$$
\kappa_{i,j}+\kappa_{j,k}+\kappa_{k.i} \leq  3\kappa_{i,j,k} +\frac{9n}{20}\;.
$$
	Therefore, up to relabelling $i,j,k$, so that $\kappa_{i,j}$ is the smallest term in the left-hand side, we obtain
$$
~~~~~~~~~~~~~~~ 
~~~~~~~~~~~~~~~ 
\kappa_{i,j} \leq  \kappa_{i,j,k} +\frac{3n}{20}\;.
~~~~~~~~~~~~~~~ 
~~~~~~~~~~~~~~~ 
\qedhere
$$
\end{proof}

\begin{rem}It is natural to expect that by using the Euler-Poincar\'e formula for rational homology spheres of higher dimensions one could obtain variants of Lemmas~\ref{lemma:3count} and~\ref{lemma:4count} with gradual improvements of the constants $\frac{1}{6}, \frac{3}{20}$ as the dimension gets higher. However this does not seem to be the case -- at least not without new ideas. Similarly, we have not been able to obtain any improvement by looking at the whole set of Dehn-Sommerville equations rather than only the Euler-Poincar\'e formula, in any dimension.
\end{rem}

\begin{lemma}\label{lemma:treecount}
	Let $C$ be a $2$-colourful graph on $[1..n]$ with $\kappa_{1,2}=c$ connected components. Then the number of $3$-colourful planar graphs $G$ on $[1..n]$ with $k$ components such that $G_{\{1,2\}}=C$
	is at most $2^{5n} n^{c-k}$.
\end{lemma}

\begin{proof}
We will bound the number of graphs $G$ satisfying the required properties by showing how to construct such graphs in two steps.
	In Step 1, we bound the number of ways to construct a minimal subgraph $H$ of $G$ that contains $C$ and has the same connected components as $G$ (hence $k$ connected components). Then in Step 2, we bound the number of ways to extend $H$ to $G$, preserving planarity.
We let $\ell_1,\dots, \ell_{c}$ be the lengths of the $\{1,2\}$-cycles in $C$. 



	{\it Step 1.}
	The subgraph $H$ consists of $C$ together with $k-c$ extra edges that connect components of $C$ together. 
	We encode  $H$ using a labelled plane forest $F$ on $[1..c]$ with $k$ components, a binary string $w_1$ of length $n$ and a string $w_2\in [1..\ell_1]\times\dots\times [1..\ell_c]$. The vertices of the forest correspond to the cycles of $C$, say ordered by increasing minimum vertex, and the edges of the forest determine which cycles are connected together with edges of colour $3$ in $H$.
	
	We use $w_1$ and $w_2$ to specify the attachment of the edges of colour $3$ between cycles as follows. 
	We explore $F$ component by component using a clockwise Depth-First-Search (DFS), using the minimum vertex yet unexplored as root for each new component. At the same time, we add edges of colour $3$ to $C$ as follows. The word $w_1$ indicates which vertices of $[1..n]$ are adjacent to an edge of colour $3$ in $H$. The word $w_2$ specifies the vertex to which an edge of colour $3$ connects the first time a cycle is visited in the DFS exploration.
 Note that once we have attached the first edge of colour $3$ to a cycle, this fixes the order in which the other edges of colour $3$ appear along the same cycle. 

Clearly, given the choice of a plane forest and two words, there is at most one way to connect the cycles in $C$ with edges of colour $3$ that is compatible with them.  

To bound from above the number of ways to construct $H$, we bound the total number of encodings. There are at most $2^n$ choices for $w_1$ and $\prod_{i=1}^c \ell_i$ choices for $w_2$. From the arithmetic-geometric mean inequality and since $\sum_{j=1}^{c} \ell_j = n$, it follows that 
$$
\prod_{i=1}^c \ell_i \leq \left(\frac{\sum_{j=1}^c \ell_j}{c}\right)^c = \left(\frac{n}{c}\right)^c
\leq \binom{n}{c} \leq 2^n\;.
$$
	Moreover, the number of labelled plane forests with $m$ vertices and $k$ components is at most $\frac{m!}{k!} {2m-k \choose m}$. Indeed, an unlabelled plane forest with $k$ ordered components and $m$ vertices can be encoded by its \L ukasiewicz path (see e.g.~\cite[Chap.~5]{Stanley:EC2}) which is a path with $m-k$ up-steps and $m$ down steps.  The number of inequivalent ways to label the vertices is at most $m!$, and once vertices are labelled all the $k!$ possible orderings of components are inequivalent.

	Since $m!/k!=m(m-1)\dots (k+1) \leq m^{m-k}$, 
this shows that the number of choices for $H$ is at most $2^{3n}n^{c-k}$ and concludes the discussion on Step~1.


	{\it Step 2.}
Given the choice of $H$, we now bound the number of ways to extend it to a $3$-colourful planar graph $G$ with the same number of components. We first add a new half-edge to each vertex of $H$ not already incident to an edge of colour $3$. The edges of ${G-H}$ can be seen as a perfect matching on these half-edges.
Now, because $G$ is planar and has the same number of components as $H$, the edges of this perfect matching form a non-crossing arch system around each face of $H$ (as $H$ is a 3-edge-coloured cubic graph, by Remark $1$ it has a canonical embedding and its faces are well-defined). The matching of these half-edges  can therefore be encoded by a well-formed parenthesis word of length equal to the number of half-edges.

	This shows that the number of ways to construct $G$ from $H$ is at most  $2^{2n}$.

	\smallskip 
The lemma follows from the bounds obtained in Step 1 and Step 2.
\end{proof}

\subsection{Induction on dimension and proof of upper bounds in Theorem~\ref{thm:main}}\label{sec:UB}

The upper bounds in Theorem~\ref{thm:main}  easily follow from the implication $(ii)\Rightarrow (iii)$ in Proposition~\ref{prop:criteria} together with Lemmas \ref{lemma:3count}, \ref{lemma:4count} and \ref{lemma:treecount}, as we now show. In fact, we are going to show a more general result.
\begin{definition}
	For $d\geq 2, n\geq 2$, we let $\mathcal{N}_d(n)$ be the set of $(d+1)$-colourful graphs on $[1..n]$ having the property that for each $\cI\subseteq [1..d]$ with $|\cI|\in \{3,5\}$ and $|\cI|\leq d$, the space $|X(G_\cI)|$ is a disjoint union of rational homology spheres.
	
	We let $\mathcal{S}_d(n)$ be the set of graphs such that $|X(G)|$ is a manifold \emph{and} a rational homology sphere. 
\end{definition}

	We let $N_d(n)=|\mathcal{N}_d(n)|$ and $S_d(n)=|\mathcal{S}_d(n)|$. As $M_d(n)\leq N_d(n)$, the upper bounds in Theorem~\ref{thm:main} follow directly from the next result.

\begin{thm}\label{thm:upperBounds}
	 For all $d\geq 3$ with $d\neq 4$ we have  
		\begin{align}\label{eq:upperBoundN}
			\frac{1}{n!}N_d(n)  \ll n^{\left(\frac{1}{6}+(d-3) \frac{3}{20}\right) n},
		\end{align}
	and for $d=4$ we have 
	\begin{align}\label{eq:upperboundN4}
		\frac{1}{n!}N_4(n)  \ll n^{\frac{n}{3} }.
	\end{align}
		Moreover for any $d\geq 3$ we have	
		\begin{align}\label{eq:upperBoundS}
			\frac{1}{n!}S_d(n)  \ll n^{\left(\frac{1}{6}+(d-3) \frac{3}{20}\right) n}.
		\end{align}
\end{thm}

We first state a lemma.
\begin{lemma}\label{lemma:inductionUpperBound}
Suppose we are given, for each $d\geq 2$ and $n\geq 2$, a family $\mathcal{A}_d(n)$ of $(d+1)$-colourful graphs on $[1..n]$  such that:
	\begin{itemize}
		\item[(a)] For any $d\geq 3$, any $G\in \mathcal{A}_d(n)$, and any colour $i\in [1..d+1]$, the $d$-colourful graph obtained from $G$ by removing all edges of colour $i$ belongs to $\mathcal{A}_{d-1}(n)$;
		\item[(b)] For each $\cI\subseteq [1..d+1]$ with $|\cI|=3$, the embedded graph $G_\cI$ is plane;
		\item[(c)] There exists a sequence $(a_d)_{d \geq 3}$ such that for any $d\geq 3$ and any $G\in \mathcal{A}_d(n)$, there exists a triple of distinct colours $i,j,k\in [1..d+1]$ such that $\kappa_{i,j}-\kappa_{i,j,k} \leq a_d n$.
	\end{itemize}
	Then we have
	$$\frac{1}{n!} |\mathcal{A}_d(n)| \ll n^{(a_3+a_4+\dots+a_d)n}$$
\end{lemma}
\begin{proof}[Proof of Lemma~\ref{lemma:inductionUpperBound}]
	The proof proceeds by induction on $d\geq 2$. 
	
	For $d=2$, Property (b) implies that $G$ is planar.
	Now, the number of planar cubic multigraphs with $n$ labelled vertices is at most $n!K^n$ for some constant $K$. Indeed, the number of unlabelled embedded planar multigraphs with $3n$ edges is bounded by $K^n$ for some $K>0$, which follows either by a spanning tree argument or by exact counting formulas~\cite{Tutte} as recalled in the introduction. Here since we have $n$ vertices, the labelling multiplies by at most $n!$. Moreover for each graph, there is an exponential number of colourings of the edges using $3$ colours. It follows that $|\mathcal{A}_2(n)|\ll n^n$.

	Let $d\geq 3$, and let $G\in \mathcal{A}_d(n)$. Take any triple of colours $i,j,k$ as in Property (c). The graph $G$ is the union of the graph $G_{[1..d+1]\setminus\{k\}}$ and the graph $G_{\{i,j,k\}}$. By induction and Property (a), there are at most $n^{(a_3+a_4+\dots+a_{d-1})n}$ choices for the first graph. By Property (b), the second graph is planar, so by Lemma~\ref{lemma:treecount} there at most $n^{a_d n}$ ways to choose it once the edges of colour $i$ and $j$ have been placed.
\end{proof}

\begin{proof}[Proof of Theorem~\ref{thm:upperBounds} and of the upper bound in~\eqref{eq:upperBoundH}]
	We apply Lemma~\ref{lemma:inductionUpperBound} with several different sequences $\mathcal{A}_d(n)$ depending on which bound we want to obtain.

	We first choose $\mathcal{A}_d(n) = \mathcal{N}_d(n)$ for $d\neq 4$ and $\mathcal{A}_4(n)=\mathcal{S}_4(n)$. This sequence satisfies Property (a) and (b) by definition. We can then take $a_3=\frac{1}{6}$ by Lemma~\ref{lemma:3count}, and $a_d=\frac{3}{20}$ for $d\geq 4$ by Lemma~\ref{lemma:4count}. This implies both~\eqref{eq:upperBoundN} and the case $d=4$ of~\eqref{eq:upperBoundS}.
		
	If we choose $\mathcal{A}_d(n)=\mathcal{N}_d(n)$ for all $d$, we can take $a_3=a_4=\frac{1}{6}$ and $a_d=\frac{3}{20}$ for $d\geq 5$, which gives a worse bound in general but is the best we can do for $d=4$, proving~\eqref{eq:upperboundN4}.

	If we choose $\mathcal{A}_d(n)=\mathcal{H}_d(n)$ for all $d$, then by Lemma~\ref{lemma:3count} we can take $a_d=\frac{1}{6}$ for all $d$, hence proving~\eqref{eq:upperBoundH}.

	It only remains to prove~\eqref{eq:upperBoundS} for $d\neq 4$. But this is a direct consequence of~\eqref{eq:upperBoundN} since $\mathcal{S}_d(n)\subseteq \mathcal{N}_d(n)$.

	\end{proof}

\section{Lower bounds}

\label{sec:lowerBounds}

\begin{figure}
	\centering
		\includegraphics[width=0.2\linewidth]{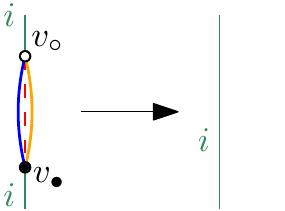}
		\caption{The operation of dipole removal.}\label{fig:dipoleRemoval}
\end{figure}

\begin{figure}
	\centering
		\includegraphics[width=\linewidth]{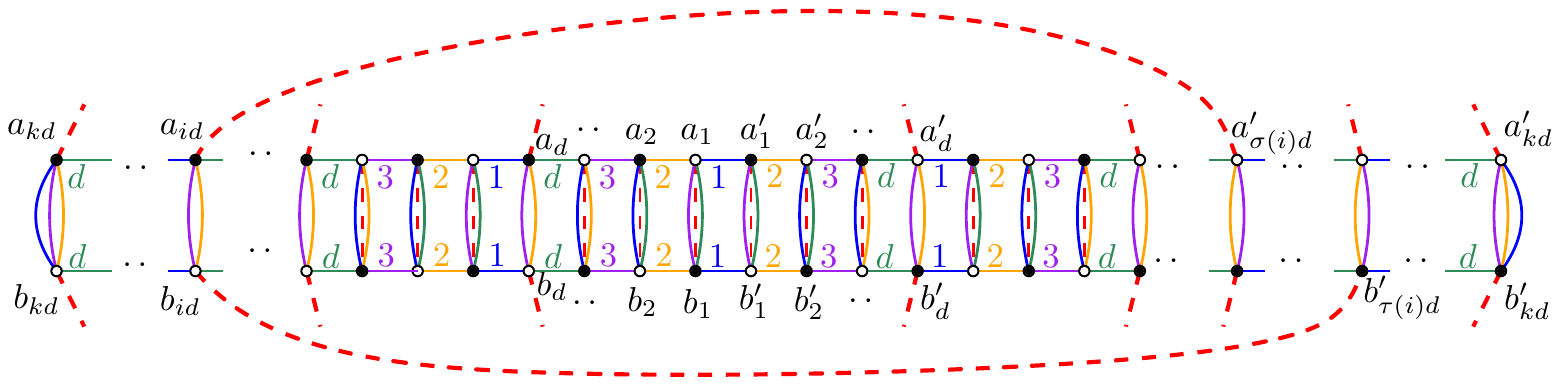}
		\caption{The lower bound construction, pictured here for $d=4$. The colour in $[1..d]$ of horizontal edges is indicated, and edges of colour $d+1$ are represented with dashed lines.}\label{fig:lowerBounddmanifolds}
\end{figure}

The goal of this section is to construct many triangulated $d$-dimensional manifolds.
We will use the implication $(i)\Rightarrow (ii)$ in Proposition~\ref{prop:criteria} to ensure that the complexes we construct are manifolds. In order to ensure that the residues are spheres, we use the criterion in the lemma below.
We first define the operation of \emph{dipole removal}, which takes as input a $(d+1)$-colourful graph $G$ and two vertices $v_\circ$ and $v_\bullet$ of $G$ that are linked together by exactly $d$ edges, using all colours in $[1..d+1]$ except some colour $i$.
The output of the dipole removal is the graph obtained from $G$ by removing the vertices $v_\circ$ and $v_\bullet$ and all the edges between them, and merging the two edges of colour $i$ incident to them into a new edge. See Figure~\ref{fig:dipoleRemoval}. The following lemma is easy and classical, see e.g.~\cite{Bonzom++}.
\begin{lemma}\label{lemma:dipoles}
	Let $d\geq 2$ and let $G$ be a connected $(d+1)$-colourful graph that can be reduced to a single dipole (the 2-vertex $(d+1)$-colourful graph) by a sequence of dipole removals. Then $|X(G)|$ is homeomorphic to a $d$-sphere. 
\end{lemma}
\begin{proof}
	One can notice, for example, that such graphs are planar for the embedding given by increasing (resp. decreasing) order of colours around white (resp. black) vertices, which is easily proved by induction. This implies the result by a criterion due to Ferri and Gagliardi~\cite{FerriGagliardi}.
\end{proof}
Graphs satisfying the criterion of Lemma~\ref{lemma:dipoles} are called \emph{melonic} in the physics literature.
We can now prove our lower bounds.
\begin{proof}[Proof of  lower bounds in Theorem~\ref{thm:main}]


	Let $k\in\mathbb{N}$. We first construct a multigraph $G_0=(V_0,E_0)$, with $n=4kd$ vertices. We refer to Figure~\ref{fig:lowerBounddmanifolds} for a visual support, the graph looks like a horizontal double path, and in our notation with use respectively prime and non-prime letters to denote the right and left part of the path, $a$ and $b$ for upper and lower vertices, and we use indices $P$, $V$ and $R$ for horizontal edges, vertical edges of colour in $[1..d]$, and for the remaining vertical edges of colour $(d+1)$.

Formally let $A=\{a_1, \dots, a_{kd}\}$, $A'=\{a_1', \dots, a_{kd}'\}$,  $B=\{b_1, \dots, b_{kd}\}$, $B'=\{b_1', \dots, b_{kd}'\}$ and $V_0=A\cup A'\cup B\cup B'$, thus $G_0$ has order $n=4kd$. We define the edge sets
\begin{align*}
E_P&=\{x_ix_{i+1}:\, x\in\{a,a',b,b'\}, i\in [1..kd-1]\}\cup \{a_1a_1', b_1b_1'\}, \\
E_V&=\{e_i(j)= a_i b_{i}:\, i\in [1..kd], j\in [1..d], i\notin \{j,j+1\} \,(\text{mod }d) \}\cup\{e_{kd}(1)=a_{kd}b_{kd}\}, \\
E'_V&=\{e'_i(j) = a'_i b'_{i}:\, i\in [1..kd], j\in [1..d], i\notin \{j,j+1\} \,(\text{mod }d) \} \cup\{e'_{kd}(1)=a'_{kd}b'_{kd}\},\\
E_R&=\{e_i(d+1)= a_i b_{i}:\, i\in [1..kd],  i\neq 0 \, (\text{mod }d) \}, \\
E'_R&=\{e'_i(d+1)= a'_i b'_{i}:\, i\in [1..kd],  i\neq 0 \, (\text{mod }d) \}, \\
E_0&= E_P\cup E_V\cup E'_V\cup E_R\cup E_R'\;.
\end{align*}
We colour the edges $x_ix_{i+1}\in E_{P}$ with colour $i+1\, (\text{mod } d)$, the edges $a_1a'_1$ and $b_1b_1'$ with colour $1$, and the edges $e_i(j)$ and $e'_i(j)$ with colour $j$. 

We will use the graph $G_0$ as a basis to construct many $(d+1)$-colourful graphs. Observe that the vertices $x_i$ with $x\in \{a,a',b,b'\}$ have degree $d$ if $i=0\,(\text{mod }d)$ and degree $d+1$ otherwise.

Let $\sigma,\tau$ be permutations of length $k$. The graph $G=G(\sigma,\tau)$ is constructed from $G_0$ by adding the edges $a_{id}a'_{\sigma(i)d}$ and  $b_{id}b'_{\tau(i)d}$ with colour $d+1$. Observe that $G$ is a $(d+1)$-colourful graph~(see Figure~\ref{fig:lowerBounddmanifolds}). Note that the coloured automorphism group of $G$ (automorphisms of $G$ that preserve the edge-colouring) has size at most $4$. Therefore, there are at least $(k!)^2 n!/4 \asymp  n^{n/2d} n!$ labelled $(d+1)$-colourful graphs of this form.

	It remains to show that $|X(G)|$ is a $d$-manifold. Let $x$ be a vertex of $X(G)$, we will show that the link of $x$ is a $(d-1)$-sphere using Lemma~\ref{lemma:dipoles}.
For $i\in [1..d+1]$, let $G^i$ be the subgraph of $G$ where edges of colour $i$ have been deleted. If $x$ has colour $i$ in $X(G)$, then the link of $x$ is homeomorphic to $|X(H)|$, where $H$ is a connected component of $G^i$. It suffices to prove that $|X(H)|$ is a sphere.

	We distinguish cases depending on the value of $i$. In the case $i=d+1$, we have $H=G^{d+1}$ (see Figure~\ref{fig:lowerBounddmanifolds} without the the dotted edges). This graph clearly satisfies the criterion of Lemma~\ref{lemma:dipoles}, as it can be reduced to a single dipole by successively lifting the dipoles from left to right.
Therefore $|X(H)|$ is a sphere in this case.

	In the case $i\in [1..d]$, if we remove all the edges coming from the permutations $\sigma$ and $\tau$ from the graph {$G^i$}, we obtain a set of \emph{gadgets} of size at most $2d$, see Figure~\ref{fig:LCTree2} left. If $i\neq1$ one of these gadgets contains the vertices $a_1,a_1',b_1,b_1'$ and we call it the \emph{central} gadget. In $G^i$, each non-central gadget is connected to two other gadgets through edges of colour $d+1$ that attach to vertices $x_{jd}$ for $x\in\{a,a',b,b'\}$ and $j\in [1..kd]$. Thus, $H$ is either the central gadget or a cycle of gadgets (see Figure~\ref{fig:LCTree2} right). In both cases this graph again satisfies the criterion of Lemma~\ref{lemma:dipoles}. Indeed for the central gadget this is similar to the previous case (remove dipoles from left to right), and for the cycles of gadgets one can successively remove dipoles to replace each gadget by a single edge of colour $d+1$, until only one gadget remains which is easily reduced from left to right.  Therefore $|X(H)|$ is a sphere in all cases.
 
Since the $(d-1)$-dimensional links of $X(G)$ are spheres, any smaller link also is.
It follows that $|X(G)|$ is a $d$-manifold and we conclude that $M_d(n)\succeq n^{n/2d} n!$.
\end{proof}


\begin{figure}
	\centering
		\includegraphics[width=\linewidth]{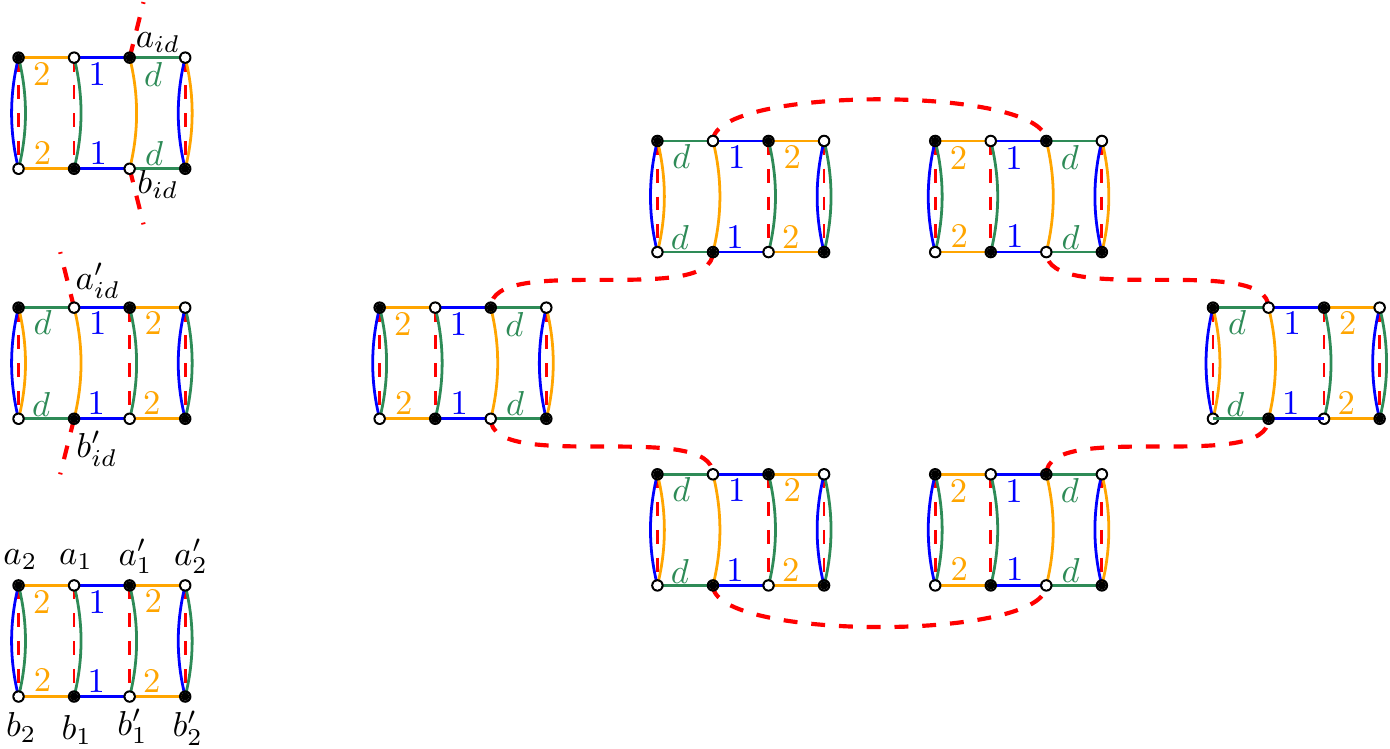}
		\caption{Left: Gadgets obtained after removing all edges of colour $i$ for an $i\in [1..d]$ (here $i=3$). Each of them has two ``outgoing'' half-edges of colour $d+1$, with the exception of the central gadget --  represented here at the bottom.  Right: After removing all edges of colour $i$, each connected component other than the central gadget is a ``cycle of gadgets''.}\label{fig:LCTree2}
\end{figure}

\section{Remaining proofs}
\label{sec:remainingProofs}

It remains to prove Theorem~\ref{thm:proba}, and the lower bound in~\eqref{eq:upperBoundH}. Both proofs are simple variants of the previous ones.

\begin{proof}[Proof of Theorem~\ref{thm:proba}]
	The theorem directly follows from the two following claims:
	\begin{itemize}[itemsep=0pt, topsep=0pt, parsep=0pt, leftmargin=20pt]
		\item[(a)] for any $c>0$ there exist $K>0$ such that the number of coloured triangulations of $3$-manifolds with $n$ labelled tetrahedra and with more that $\frac{Kn}{\log n}$ vertices is at most $2^{-c n} n^{\frac{7n}{6}} $;
		\item[(b)] there exists $c_0>0$ such that the number of coloured triangulations of $3$-manifolds with $n$ labelled tetrahedra and with less than $\frac{n}{\log n}$ vertices is at least $2^{-c_0 n} n^{\frac{7n}{6}}$.
	\end{itemize}
\medskip	

	We first prove claim (a).  Let $K>0$ (to be chosen later) and let $G$ be a $4$-colourful graph such that $X(G)$ has more than $\frac{Kn}{\log{n}}$ vertices and its underlying space is a $3$-manifold. 
	Since vertices of $X(G)$ are in bijection with connected components of $3$-coloured subgraphs of $G$, there exist distinct colours $i,j,k$ such that $\kappa_{i,j,k} \geq \frac{K n}{4 \log n}$. By~\eqref{eq:last} 
	with $\cI=\{i,j,k\}$ and up to relabelling $i,j,k$, we have
	$$
	\kappa_{i,j} -\kappa_{i,j,k} \leq \frac{n}{6} -\frac{1}{3} \kappa_{i,j,k} \leq \frac{n}{6}-\frac{Kn}{12\log{n}}.
	$$
We can upper bound the number of ways to construct such a graph $G$ as in the proof of our main theorem. We first choose the planar graph $G_{\{i,j,\ell\}}$ where $\ell\in [1..4]\setminus \cI$. There are at most $n^n$ ways to do it. As in Lemma~\ref{lemma:treecount}, once this graph has been chosen there are at most $2^{5n} n^{\kappa_{i,j} -\kappa_{i,j,k}}$ ways to place the edges of colour $k$. We thus have at most  $2^{(5-K/12)n}n^{\frac{7n}{6}} $ choices of graphs in total, which is less than $2^{-c n}n^{\frac{7n}{6}} $ provided we take $K=K(c)$ large enough.

	We now prove claim (b) by following the construction given in the proof of the lower bound of Theorem~\ref{thm:main} (see Section~\ref{sec:lowerBounds}). The $4$-coloured graph constructed in that proof depends on two permutations $\sigma$ and $\tau$ that describe the incidences of the edges of colour $d+1$. By construction,  $\kappa_{1,2,3}=1$, and for any $\cI\subseteq [1..4]$ with $|\cI|=3$ and $\cI\ni 4$, $\kappa_\cI$ is the number of cycles of the permutation $\sigma \tau^{-1}$ (or this number plus one, for choices of $\cI$ that involve the central gadget). If $\sigma$ and $\tau$ are chosen uniformly at random, the expected number of cycles of $\sigma \tau^{-1}$ is $O(\log n)$. Therefore with positive probability,
the number of vertices in $X(G)$ is $O(\log{n})$, which is smaller than $\frac{n}{\log n}$ for $n$ large enough. The claim follows since there are $ 2^{\Omega(n)}n^{7n/6}$ such graphs $G$.
\end{proof}

\begin{proof}[{Proof of the lower bound in~\eqref{eq:upperBoundH}}]
	The construction is similar to the one of Section~\ref{sec:lowerBounds} and we will only sketch it. Let $G$ be a graph obtained in our lower bound construction for $d=3$~(see Figure~\ref{fig:lowerBounddmanifolds}). For $d\geq 4$, we obtain a $(d+1)$-colourful graph 
	adding edges $a_i b_i$ and $a_i' b_i'$ for every $i\in [1..kd]$ and every colour in $\{5,\dots, d+1\}$. Note that there are $M_3(n)\succeq n^{n/6}n!$ such graphs. Using a similar analysis as the one in Section~\ref{sec:lowerBounds}, it is easy to see that for any set of colours $\cI$ of size $3$, the graph $G_\cI$ is planar.
%
\end{proof}

\subsection*{Acknowledgements}
We thank the anonymous referees for their valuable comments, and one of them for correcting a mistake in the proof of the lower bound.

%
%
%
%
%
%

\bibliographystyle{plain}
\bibliography{spheres}

\end{document}